\documentclass[12pt,a4paper]{amsart}          
\usepackage{latexsym}
\usepackage{amsfonts}
\usepackage{amsmath}
\usepackage{enumerate}

\newcommand{\RH}{\mathbb{R}H^{n}(-k^2)}
\newcommand{\CH}{\mathbb{C}H^n(-4k^2)}
\newcommand{\QH}{\mathbb{H}H^{n}(-4k^2)}
\newcommand{\OH}{\mathbb{O}H^{2}(-4k^2)}
\newcommand{\KH}{\mathbb{K}H^{n}(-4k^2)}
\newcommand{\K}{\mathbb{K}}
\newcommand{\oc}{\mathbb{O}}
\newcommand{\Q}{\mathbb{H}}
\newcommand{\C}{\mathbb{C}}
\newcommand{\R}{\mathbb{R}}
\newcommand{\CP}{\mathbb{C}P^{n}(4k^2)}
\newcommand{\h}{\mathbb{H}}
\newcommand{\vol}{\mathop{\rm vol}}

\newcommand{\arctanh}{\mathrm{arctanh}}

\theoremstyle{definition}
\newtheorem{definition}{Definition}[section]
\newtheorem{example}[definition]{Example}

\theoremstyle{plain}
\newtheorem{theorem}[definition]{Theorem}
\newtheorem{proposition}[definition]{Proposition}
\newtheorem{corollary}[definition]{Corollary}
\newtheorem{lemma}[definition]{Lemma}
\newtheorem*{theorem*}{Theorem}
\newtheorem*{theoremA}{Theorem A}
\newtheorem*{theoremB}{Theorem B}
\newtheorem*{theoremN}{Theorem}

\theoremstyle{remark}

\newtheorem{remark}[definition]{Remark}

\title{Convexity on Complex Hyperbolic Space}

\author{Judit Abardia} 
\author{Eduardo Gallego} 
\address{Departament de Matem{\`a}tiques, Facultat de Ci{\`e}ncies,
Universitat Aut{\`o}noma de Barcelona, 08193--Bellaterra (Barcelona), Spain}
\email{juditab@mat.uab.cat, egallego@mat.uab.cat}

\thanks{Work partially supported by DGICYT grant \#MTM2009-07594 and by the {\em Departament d'Innovaci\'o Universitats i Recerca} from the Generalitat of Catalonia and the ESF}


\begin{document}
\maketitle

\begin{abstract}
In a Riemannian manifold a regular convex domain is said to be $\lambda$-convex if its normal curvature at each point is greater than or equal to $\lambda$. In a Hadamard manifold, the asymptotic behaviour of the quotient $\vol(\Omega(t))/\vol(\partial\Omega(t))$ for a family of $\lambda$-convex domains $\Omega(t)$ expanding over the whole space has been studied and general bounds for this quotient are known.
In this paper we improve this general result in the complex hyperbolic space $\CH$, a Hadamard manifold with constant holomorphic curvature equal to $-4k^2$. Furthermore, we give some specific properties of convex domains in $\CH$ and we prove that  $\lambda$-convex domains of arbitrary radius exists if $\lambda\leq k$.  
\end{abstract}

\section{Introduction}
Given a family of convex domains expanding
over the whole Euclidean plane, the quotient between the area and the perimeter tends to infinity. This behaviour does not hold in hyperbolic
plane $\mathbb{H}^2(-1)$ where the quotient tends to a value less or equal
than 1.

A convex domain in $\h^2(-1)$ is said to be $h$-convex if for every pair of points in the convex domain, each horocyclic segment joining them also belongs to the convex domain. The first result about the asymptotic behavior of the quotient $\text{area}/\text{perimeter}$ of convex domains in $\mathbb{H}^2(-1)$ was stated for a family of $h$-convex domains by Santal\'{o} and Ya\~{n}ez \cite{santalo.yanez72} in 1972. 
They proved that if $\{\Omega(t)\}_{t\in\R^{+}\!}$ is a family of compact $h$-convex domains
in $\h^2(-1)$ expanding over the whole plane, then
\begin{equation}\label{areaLength}\lim_{t\rightarrow\infty}\frac{\textrm{area}(\Omega(t))}{\textrm{length}(\partial\Omega(t))}=1.\end{equation}
In hyperbolic plane, the set of equidistant points to a geodesic line are two curves called equidistant lines. These curves have constant geodesic curvature $\lambda$ such that $0<\lambda<1$. It is said that a convex domain is $\lambda$-convex if for every pair of points in the convex domain, each segment of curve with constant geodesic curvature $\lambda$ joining them also belongs to the convex domain. In \cite{gallego.reventos99} it is proved that a convex domain in $\h^2(-1)$ is $\lambda$-convex if and
only if the geodesic curvature of the boundary is greater than or equal to
$\lambda$.  In \cite{gallego.reventos99} it is also proved that the quotient in (\ref{areaLength}) can take any value between $[\lambda,1]$ for a family of $\lambda$-convex sets. Roughly speaking $\lambda$-convex sets ($\lambda>0$) have the boundary more curved than ($0\,$-)convex sets. The notion of $\lambda$-convex domain is also defined for arbitrary Riemannian manifolds (see Definition \ref{convexDomain}).

Recall that a Hadamard manifold is a simply connected Riemmanian manifold with non-positive sectional curvature. For Hadamard manifolds it was stated in \cite{borisenko.gallego.reventos} the following result concerning $\lambda$-convexity.

\begin{theoremN}[\cite{borisenko.gallego.reventos}]
Let $M$ be a Hadamard manifold with sectional curvature $K$ such that $-k^2_{2}\leq K\leq -k_{1}^2$. It can only exist families of $\lambda$-convex domains expanding over the whole space $M$ if $\lambda\leq k_{2}$.
\end{theoremN}

Complex hyperbolic space $\CH$ is a Hadamard manifold of real dimension $2n$ with real sectional curvature $K$ such that $-4k^2\leq K\leq -k^2$. 

\smallskip
In this paper we use properties of normal curvature of spheres in complex hyperbolic space (see Corollary \ref{esferes}) to prove that the estimate of the preceding theorem now becomes $\lambda\leq k$. Note that we restrict it from $k_{2}$ to $k_{1}$.

\begin{theorem}\label{theorem1}In complex hyperbolic space $\CH$, it can only exist families of compact $\lambda$-convex domains expanding over
the whole $\CH$ if $\lambda\leq k$.\end{theorem}

The result about the asymptotic behaviour of $\vol(\Omega(t))/\vol(\partial\Omega(t))$ was studied for families of $\lambda$-convex domains $\{\Omega(t)\}$ expanding over the whole $\h^n$ in \cite{borisenko.miquel99} and \cite{borisenko.vlasenko}. In \cite{borisenko.gallego.reventos} it was stated for Hadamard manifolds as follows: 

\begin{theoremN}[\cite{borisenko.gallego.reventos}]
Let $\{\Omega(t)\}_{t\in\R^+}$ be a piecewise $\mathcal{C}^2$ family of $\lambda$-convex compact domains, $0\leq\lambda\leq k_2$, in a Hadamard manifold $M$ of dimension $n$ with sectional curvature $K$ such that $-k_{2}^2\leq K\leq -k_{1}^2$ and suppose that it expands over the whole $M$. Then
$$\frac{\lambda}{(n-1)k_2^2}\leq
\liminf_{t\rightarrow\infty}\frac{\vol(\Omega(t))}{\vol(\partial\Omega(t))}\leq
\limsup_{t\rightarrow\infty}\frac{\vol(\Omega(t))}{\vol(\partial\Omega(t))}\leq
\frac{1}{(n-1)k_1}.$$
\end{theoremN}

When we apply this result to $\h^n$ we obtain the same bounds as in \cite{borisenko.vlasenko}. Moreover, these bounds are sharp (see \cite{solanes.tesi}).
The same theorem applied to complex hyperbolic space gives the following result.

\begin{corollary}Let $\{\Omega(t)\}_{t\in\R^+}$ be a piecewise $\mathcal{C}^2$ family of $\lambda$-convex compact domains, $0\leq\lambda\leq 2k$, in $\CH$ with sectional curvature $K$ such that $-4k^{2}\leq K\leq -k^2$ and suppose that it expands over the whole $\CH$. Then
$$\frac{\lambda}{4(2n-1)k^2}\leq
\liminf_{t\rightarrow\infty}\frac{\vol(\Omega(t))}{\vol(\partial\Omega(t))}\leq
\limsup_{t\rightarrow\infty}\frac{\vol(\Omega(t))}{\vol(\partial\Omega(t))}\leq
\frac{1}{(2n-1)k}.$$\end{corollary}

The second goal of this paper is to show that inequalities in this corollary can be improved. Using properties of the differential geometry of complex hyperbolic space we obtain new bounds for the asymptotic behaviour of the quotient and we prove that the new estimate for the upper bound is sharp; it is attained, among others, for a family of geodesic spheres. 

\begin{theorem}\label{theorem2}Let $\{\Omega(t)\}_{t\in\R^+}$ be a piecewise $\mathcal{C}^2$ family of $\lambda$-convex compact domains, $0\leq \lambda\leq k$, expanding over the whole space $\CH$, $n\geq 2$. Then,
\begin{equation}\label{fitesTeorema}\frac{\lambda}{4nk^2}\leq \liminf_{t\rightarrow\infty}\frac{\emph{vol}(\Omega(t))}{\emph{vol}(\partial\Omega(t))}\leq \limsup_{t\rightarrow\infty}\frac{\emph{vol}(\Omega(t))}{\emph{vol}(\partial\Omega(t))}\leq \frac{1}{2nk}.\end{equation} Moreover, the upper bound is sharp.\end{theorem}

This paper is divided into five sections. In Section \ref{seccio2} we review the definition and some properties we need about complex hyperbolic space. We also define, more precisely, the notion of $\lambda$-convexity in Hadamard manifolds. In Section \ref{seccio3} we give a proof of Theorem \ref{theorem1}, we review some examples of $\lambda$-convex domains in complex hyperbolic space and we prove that tubes along a geodesic are convex domains in a Hadamard manifold. Section \ref{seccio4} is devoted to the proof of the main Theorem \ref{theorem2}. Finally, Section \ref{seccio5} contains some final remarks: we extend Theorem \ref{theorem1} and Theorem \ref{theorem2} in any non-compact rank one symmetric space and we study the quotient volume/area for a family of tubes described in Section \ref{seccio3}.

\section{Preliminaries}\label{seccio2}

\subsection{Complex hyperbolic space} 
\begin{definition}\emph{Complex hyperbolic space} is the only (up to holomorphic isometries) complete simply connected K\"{a}hler manifold of constant holomorphic curvature $-4k^2$. We denote it by $\CH$.
\end{definition}

Consider on $\C^{n+1}$ the Hermitian product $$\langle
z,w\rangle=-z_0\overline{w_0}+\sum_{j=1}^{n}z_j\overline{w_j}$$
and the subset $$M=\{z\in\C^{n+1}\,\,:\,\,\langle
z,z\rangle<0\}.$$ 
The projection of $M$ into $\mathbb{CP}^n$ corresponds to complex hyperbolic space $\CH$. We will consider $\CH$ equipped with the Bergman metric (see \cite{goldman} pp.68-74) rescaled in a way such that sectional curvature lies in $[-4k^2,-k^2]$. 
Geodesics in $\CH$ are projection of complex planes in $\C^{n+1}$ and the isometry group is $PU(1,n)$.

\smallskip
The \emph{holomorphic sectional curvature} $K_{hol}(v)$ for a unit tangent vector $v$ in a K\"ahler manifold is the sectional curvature of the plane generated by $\{v,Jv\}$. A K\"ahler manifold is said to be of constant holomorphic curvature if $K_{hol}(v)$ does not depend on $v$. A space with constant holomorphic curvature (a complex space form) is locally isometric to complex Euclidean space $\C^n$, complex projective space $\CP$, or complex hyperbolic space $\CH$. For these spaces sectional curvature of planes spanned by orthonormal vectors $u$, $v$ is given by 
$$g(R(u,v)v,u)=\frac{K_{hol}}{4}(1+3(g(u,Jv))^2)$$ (see \cite{kobayashi.nomizu} p.167).
If we consider complex hyperbolic space with constant holomorphic curvature $-4k^2$ then, from the equation above, its sectional curvature $K$ lies in $[-4k^2,-k^2]$.

\smallskip
In $\CH$ geodesic spheres are not umbilical hypersurfaces as  they are in real space forms. The knowledge of the explicit value of its principal curvatures shall be essential in the proof of the first theorem. 

\begin{proposition}[\cite{montiel}]\label{principalsEsfera}
In $\CH$ the principal curvatures of a geodesic sphere of radius $r$ are:
\begin{enumerate}[a)]\item $2k\coth(2kr)$ with multiplicity 1 and
principal direction $-JN$ (where $N$ is the inward unit normal
vector).
\item $k\coth(kr)$ with multiplicity $2n-2$.
\end{enumerate}
\end{proposition}
Furthermore,
\begin{corollary}\label{esferes}Normal curvature of geodesic spheres
takes all possible values in the interval $[k\coth(kr), 2k\coth(2kr)]$.
\end{corollary}

From Proposition \ref{principalsEsfera} we prove the following result:
\begin{corollary}\label{totgeod}
Let $z$ be a point in a geodesic sphere in $\CH$, let $N$ be the inward
unit normal vector and let $v$ be a principal direction. Then the
submanifold generated by the exponential map on $\mathrm{span}\{N,v\}$ at $z$ is totally geodesic in $\CH$.
\end{corollary}
\begin{proof}
If $v=-JN$, then the submanifold generated by the exponential map on $\mathrm{span}\{N, JN\}$ is isometric to $\h^2(-4k^2)$ which is totally geodesic in $\CH$.

If $v\neq-JN$, then $\langle v,JN\rangle=0$ since $v$ and $JN$ are principal directions of different principal curvatures, but this is the condition that vectors $\{N,v\}$ have to satisfy to generate a totally real plane. So, the submanifold generated by the exponential map is isometric to $\h^2(-k^2)$ and also totally geodesic in $\CH$.
\end{proof}

\subsection{$\lambda-$convexity in Hadamard manifolds}
Let us recall the definition of $\lambda$-convexity in a Hadamard manifold as it is stated in \cite{borisenko.gallego.reventos}.
 
\begin{definition}\label{l-convexity}
A \emph{$\mathcal{C}^2$ hypersurface} of a Riemannian manifold is
said to be \emph{regular $\lambda$-convex} if at every point all
the normal curvatures (with respect to the inward unit normal
vector) are greater than or equal to $\lambda\geq 0$.
\end{definition}
This definition can be generalized to the non-regular case.
\begin{definition}\label{noC2}
A \emph{$\lambda$-convex hypersurface} is a hypersurface such that
for every point $p$ there is a regular $\lambda$-convex
hypersurface $S$ leaving a neighborhood of $p$ in the hypersurface
in the convex side of $S$. 
\end{definition}
\begin{definition}\label{convexDomain}
A domain is said to be \emph{(regular) $\lambda$-convex} if its
boundary is a (regular) $\lambda$-convex hypersurface.
\end{definition}

\begin{remark}Note that a $\lambda$-convex hypersurface is also $\lambda'$-convex for any $\lambda'\leq\lambda$.\end{remark}

\begin{remark}In spaces of constant sectional curvature and in complete
simply connected manifolds with non-positive sectional curvature
the notion of 0-convexity is equivalent to the convexity with
respect to geodesics (cf. \cite{alexander}).
\end{remark}

\section{$\lambda-$convexity in complex hyperbolic space}\label{seccio3}

\begin{definition}Let $\Omega\subset\CH$ be a compact convex domain and let $p\in\partial\Omega$ be a point of class $\mathcal{C}^1$. The largest  (smallest) sphere tangent to $\partial\Omega$ at $p$ contained in (containing) the domain is called \emph{inscribed (circumscribed) sphere at $p$}.\end{definition}

\begin{lemma}\label{normal}
Let $\Omega\subset\CH$ be a piecewise $\mathcal{C}^2$ convex domain. If at $p\in\partial\Omega$ there exists inscribed sphere $S_i$ and circumscribed sphere $S_c$, then
\begin{eqnarray}\label{normalCurvatures}K_{n,S_c}(v)\leq K_{n,\partial\Omega}(v)\leq
K_{n,S_i}(v)\end{eqnarray} for any $v$ principal direction of the inscribed sphere at $p$. ($K_{n,\,\cdot\,}(v)$ denotes the normal curvature  in the direction $v$ at $p$.)
\end{lemma}
\begin{remark}If at a point $p$ there exists only circumscribed or inscribed sphere, then the corresponding inequality remains true.
\end{remark}
\begin{proof}
Let $N$ be the inward unit normal vector to $\partial\Omega$ at $p$. Let
$\{e_1,...,e_{2n-1}=-JN\}$ be a basis of principal directions of the
inscribed sphere at $p$. Note that it is also a basis of principal
directions of the circumscribed sphere at $p$. Indeed, the direction $-JN$ is always a principal direction of a sphere and the directions perpendicular to $-JN$ are all principal directions (cf. Proposition \ref{principalsEsfera}). 

From Corollary \ref{totgeod}, the submanifold generated by the exponential map on $\text{span}\{N,e_i\}$ at $p$ is totally geodesic in $\CH$. When
$i=2n-1$ it is isometric to $\h^2(-4k^2)$. In the other cases it is isometric to $\h^2(-k^2)$.

Clearly, in real hyperbolic plane the inequality (\ref{normalCurvatures}) is satisfied. That is, suppose $p$ is a point in the boundary of a convex domain in a real hyperbolic plane. If there exists circumscribed and inscribed sphere at $p$, then geodesic curvature of the convex domain at $p$ lies between geodesic curvature of the spheres at $p$. Then, inequality (\ref{normalCurvatures}) is also valid for planes generated by the normal direction and a principal direction of the inscribed (circumscribed) sphere. 
\end{proof}

\begin{proposition}\label{bounds}
Let $\Omega\subset\CH$ be a piecewise $\mathcal{C}^2$ compact convex
domain and $N$ the inward normal vector of $\partial\Omega$. If at  $p\in\partial\Omega$ of class $\mathcal{C}^2$ there exists inscribed and circumscribed sphere, then the normal curvature
in directions $v$ such that $\langle v,JN\rangle=0$ satisfies
$$k\coth(kR)\leq K_{n,\partial\Omega}(v)\leq k\coth(kr)$$ and in the
direction $-JN$ satisfies
\[2k\coth(2kR)\leq K_{n,\partial\Omega}(-JN)\leq2k\coth(2kr)\] with $r$ the
radius of the inscribed sphere at $p$
and $R$ the radius of the circumscribed sphere at $p$.
\end{proposition}
\begin{proof}Let us restrict to the first case. If $v\in T_{p}\partial\Omega$ such that $\langle v,Jn\rangle=0$, then 
$v$ is a principal direction of the inscribed and circumscribed sphere at $p$. The normal curvature of the inscribed (resp. circumscribed) sphere at $p$ in the direction $v$ is $k\coth(kr)$ (resp. $k\coth(kR)$) (see Proposition \ref{principalsEsfera}). By Lemma \ref{normal} we deduce the inequalities. 
\end{proof}

Now using this proposition we shall prove Theorem \ref{theorem1}. 

\begin{proof}[Proof of Theorem 1]
Let $\Omega(t)$ be any convex domain of the family and let $p\in\partial\Omega(t)$ be of class $\mathcal{C}^2$ with inscribed sphere $S_{i}(r)$. It follows from the definition of $\lambda$-convexity and Proposition \ref{bounds} that for some direction $v$ tangent to $\partial\Omega(t)$ at $p$ it is satisfied
\begin{equation}\label{d}
\lambda\leq K_{n,\,\partial\Omega(t)}(v)\leq K_{n,S_i}(v)=k\coth(kr).
\end{equation} 

Since the radius of the inscribed ball tends to infinity when the convex domain grows, for a convex domain big enough, $k\coth(kr)$ is close
to $k$. From inequalities in (\ref{d}) it is necessary that
$\lambda\leq k$.
\end{proof}

\begin{remark}
The proof of this theorem uses properties of $\CH$ that are not true in a general Hadamard manifold. That is, we use the fact that for each point $p$ in a geodesic sphere there are directions of the tangent space with normal curvature equal to $k\coth(kr)$. Given a Hadamard manifold with sectional curvature $K$ such that $-k_2^2\leq K\leq -k_1^2$ the normal curvature $K_n$ of a sphere of radius $r$ satisfies $k_1\coth(k_1r)\leq K_n\leq k_2\coth(k_2r)$ but the equality in (\ref{d}) may not be taken (cf. \cite{petersen}).
\end{remark}

\begin{example}
\emph{Geodesic spheres}. In any Hadamard manifold geodesic spheres are compact convex domains. In $\CH$ the normal curvature of
geodesic spheres of radius $r$ varies from $k\coth(kr)$ to $2k\coth(2kr)$ (cf. Proposition \ref{principalsEsfera}). Then, they
are $k\coth(kr)$-convex domains. 

\emph{Horospheres}. If we fix a point in a sphere and
let the radius tend to infinity we get a non-compact convex
domain called \emph{horosphere}. It is a $k$-convex domain.

\emph{Equidistant hypersurface}. Let $\mathbb{C}H^{p}(-4k^2)$, $1\leq p<n$, be an isometrically embedded complex hyperbolic space in $\CH$ and let $\RH$ be an  isometrically embedded real hyperbolic space in the complex hyperbolic space $\CH$. The \emph{equidistant hypersurface} from a $\mathbb{C}H^{p}(-4k^2)$,
$1\leq p<n$, or from a $\RH$ consists of all points at distance $r$ from the submanifold. Equidistant hypersurfaces are also convex hypersurfaces bounding a non-compact domain. Its $\lambda$-convexity is $k\tanh(kr)$. These facts follow from \cite{montiel}.

\emph{Bisectors}. It is known that in $\CH$ do not exist totally geodesic (real) hypersurfaces. The hypersurfaces which are considered as the substitute of totally geodesic hypersurfaces in $\CH$ are the ones generated by
the exponential map on a $(2n-1)$-dimensional tangent space at a point, that
is, \emph{bisectors}. However, they are not convex hypersurfaces
and do not bound convex domains (cf. \cite{goldman} p.193). So, we
cannot construct any convex domain with part of its boundary
contained in a bisector.
\end{example}

\subsection{Tube along a geodesic} Other examples of compact convex domains in $\CH$ can be constructed from the tube of radius $r$ about a geodesic.
\begin{definition}
Let $(M,g)$ be a Hadamard manifold and $\gamma$ a complete geodesic. The \emph{tube} $\tau(\gamma,r)$ of radius $r\geq 0$ about $\gamma$ is defined as
\begin{eqnarray*}
\tau(\gamma,r)&=&\{x\in M\,\,:\,\,\exists\textrm{ a geodesic segment }\xi\textrm{ of length }
L(\xi)\leq r \\ & &\quad \textrm{ from }x\textrm{ meeting }\gamma\textrm{
orthogonally}\}.
\end{eqnarray*}
\end{definition}
\begin{remark}As $\gamma$ is complete, the definition of $\tau(\gamma,r)$ is equivalent to (cf. \cite{gray} p.32) $$\bigcup_{y\in\gamma}\{\exp_y(v)\,\,:\,\,v\in(\gamma_y)^{\bot}\textrm{ and }||v||\leq r\}$$
and also to
$$\bigcup_{y\in\gamma}\{\exp_y(v)\,\,:\,\,v\in T_yM\textrm{ and }||v||\leq
r\}.$$ That is, a tube about a complete geodesic $\gamma$ can be defined as
the set of points at a distance from $\gamma$ less or equal than a fixed $r\geq 0$ (the radius of the tube) or as the union of the balls of radius $r$ with center in $\gamma$.
\end{remark}

In the next proposition we prove that, in any Hadamard manifold, the tube about a geodesic is a convex set and in the next corollary we give a way to modify this tube to obtain a convex compact domain.

\begin{proposition}Let $(M,g)$ be a Hadamard manifold and let $\gamma$ be a complete geodesic in $M$. The tube of radius $r$ about $\gamma$ is a convex domain.
\end{proposition}
\begin{proof}
Let $t_0\in\R$ and let $v_0\in(T_{\gamma(t_0)}\gamma)^{\bot}$ be such that $p=\phi(r,t_0,v_0)=\exp_{\gamma(t_0)}(rv_0)$. Let $v(t)$ denote the parallel transport of $v_0$ along $\gamma$ and let us consider the curve $\alpha(s)=\phi(r,s,v(s))=\exp_{\gamma(s)}(rv(s))$ on $\partial\tau(\gamma,r)$. Then $\alpha'(0)$ is a tangent vector of $\partial\tau(\gamma,r)$ at $p$. The inward unit normal vector at $p$ is given by $N=-d\phi(\partial/\partial r)=-\partial/\partial r(\exp_{\gamma(t_0)}(rv_0))$. In the following we prove that the normal curvature at  $p\in\partial\tau(\gamma,r)$ along the direction $T=\alpha'(0)=d\phi(\partial/\partial t)$ is non-negative. We have
\begin{align*}K_n(T)&=\langle-\nabla_{T/||T||}N,T/||T||\rangle\\&=\frac{1}{||T||^2}\langle[d\phi\partial/\partial t,d\phi\partial/\partial r]-\nabla_NT,T\rangle=\frac{-1}{2||T||^2}N\langle T,T\rangle\nonumber.\end{align*}
Note that $T$ is a Jacobi field with initial conditions $T(0)=v$ and $T'(0)=0$. Thus it can be expressed as
$$T=d\phi_{(r,t,v)}(\partial/\partial t)=\left.\frac{\partial\exp_{\gamma(t+s)}(rv(s))}{\partial s}\right|_{(r,0)}.$$

In a Riemannian manifold, Jacobi fields $Z$ along a geodesic $\beta$ satisfy the equation $$\nabla_{\beta'}\nabla_{\beta'}Z+R(Z,\beta')\beta'=0.$$ From this equation we prove that in a Hadamard manifold the function $-N\langle T(r),T(r)\rangle$ is increasing: \begin{eqnarray}\frac{-\partial}{\partial r}(N\langle T,T\rangle)&=&\frac{\partial^2}{\partial r^2}\langle T(r),T(r)\rangle\nonumber\\&=&2\langle\nabla_{\partial\phi/\partial r}\nabla_{\partial\phi/\partial r}T(r),T(r)\rangle+\langle\nabla_{\partial\phi/\partial_r}T(r),\nabla_{\partial\phi/\partial_r}T(r)\rangle\nonumber\\&=&-2\langle R(T(r),\partial\phi/\partial_r)\partial\phi/\partial_r,T(r)\rangle+||\nabla_{\partial\phi/\partial_r}T(r)||^2\geq 0.\nonumber\end{eqnarray}
Moreover, the function $-N\langle T(r),T(r)\rangle$ is zero when $r$ tends to 0. That is, \begin{eqnarray}-N\left.\langle T(r),T(r)\rangle\right|_{r=0}&=&2\left.\langle\frac{\partial}{\partial r}\right|_{r=0}(d\exp_{\gamma(t_0)}rv_0)\gamma'(t_0),\gamma'(t_0)\rangle\nonumber\\&=&2\langle\left.\frac{\partial}{\partial s}\right|_{s=0}\lim_{r\rightarrow0^+}(d\exp_{\gamma(t_0+s)}rv(s))v(s),\gamma'(t_0)\rangle\nonumber\\&=&2\langle\frac{\partial}{\partial s}v(s),\gamma'(t_0)\rangle=0\nonumber\end{eqnarray} because $v(s)$ is the parallel transport of $v(t_0)$ along $\gamma(s)$.

Now we describe $2n-2$ curves by $p$ with tangent vector at $p$ orthonormal to $v_{0}$.

Let $\{v_1,...,v_{2n-2}\}$ be orthonormal vectors of $T_{\gamma(t_0)}M$ orthogonal to $\gamma'(t_0)$ and $v_0$ and let us consider the curves $\beta_j(s)=\phi(r,t_0,v_0+sv_j)=\exp_{\gamma(t_0)}(r(v_0+sv_j))$ on $\partial\tau(\gamma,r)$. Then $\beta_j'(0)$ is a tangent vector of $\partial\tau(\gamma,r)$ at $p$ and $\{\alpha'(0),\beta_1'(0),...,\beta_{2n-2}'(0)\}$ is a basis of $T_p\partial\tau(\gamma,r)$.

As in the previous case we shall prove that the normal curvature at each direction $T_j=\beta_j'(0)=d\phi(\partial/\partial s)$ is non-negative. Fields $T_j$ are also Jacobi fields but with initial conditions $T_j(0)=0$ and $T_j'(0)=v_j$. We can write them as \[T_j=d\phi_{(r,t,v)}(\partial/\partial{v_j})
=\left.\frac{\partial\exp_{\gamma(t)}(r(v+sv_j))}{\partial/\partial{s}}\right|_{(r,0)}.\]
Therefore, $-N\langle T_j(r),T_j(r)\rangle$ is increasing for every $j\in\{1,...,2n-2\}$ and it is zero when $r$ tends to 0 because $\partial/\partial s(v_0+sv_j)=v_j$ and $\langle v_j,\gamma'(t_0)\rangle=0.$ 

Thus, in a Hadamard manifold tubes of any radius about a geodesic are convex domains. 
\end{proof}

\begin{remark}\label{altresTubes}In order to obtain compact convex domains we can, for example, intersect a tube with a ball centered in the geodesic with radius greater than the radius of the tube. 
\end{remark}

\begin{remark}In complex hyperbolic space the tube $\tau(\gamma,r)$ about a geodesic segment is not a convex domain since a part of its boundary is contained in a bisector.\end{remark}

\begin{definition}\label{modifiedTube}
Let $\gamma_L$ be a geodesic segment of length $L$. We define the \emph{modified tube $\tau_{L}(\gamma,r)$ of radius $r\geq 0$ about $\gamma_{L}$} as  the set of points at distance $r$ of $\gamma_{L}$.
\end{definition}

\begin{corollary}\label{tubModificat}
Let $(M,g)$ be a Hadamard manifold and let $\gamma_{L}$ be a geodesic segment of length $L$ in $M$. The modified tube $\tau_L(\gamma,r)$ is a convex domain.
\end{corollary}
\begin{proof}
By Definition \ref{noC2} the modified tube is a convex domain if for every point in its boundary there exists a convex hypersurface supporting $\tau_{L}(\gamma,r)$.

The modified tube is the union of the balls of radius $r$ and center in the geodesic segment. The points of the intersection between a ball centered in one of the points of $\gamma_{L}$ and the tube $\tau(\gamma,r)$ are clearly supported by the tube about the complete geodesic $\gamma$. The other boundary points are also in the boundary of the ball centered in one of the endpoints of $\gamma_{L}$. Hence $\tau_L(\gamma,r)$ is a convex set.
\end{proof}

The modified tube will be useful to calculate explicitly its volume in the complex hyperbolic space (see Section \ref{seccio5}).

\section{Asymptotic behaviour}\label{seccio4}
Let $p$ be a fixed point in the interior of $\Omega$ and let $T'_{p}\CH$ be the unit tangent space at $p$. If $u\in T'_{p}\CH$ and $l(u)$ is the distance between $p$ and $\partial\Omega$ in the direction $u$, then the exponential map at $p$ on the set $$A=\{(u,t)\in T'_{p}\CH\times\mathbb{R}\,\,|\,\,0< t\leq l(u)\}$$ parametrizes $\Omega$ and
$$\textrm{vol}(\Omega)=\int_\Omega\tau=
\int_A\exp_p^{\ast}\tau$$ 
where $\tau$ is the volume element of $\CH$.

If $\mathcal J(t,u)$ is the Jacobian of the exponential map, then $\exp_p^\ast\tau=\mathcal{J}(t,u)t^{2n-1}dtdS_{2n-1}$. Using an orthonormal basis and the corresponding Jacobi fields along the geodesic given by
the direction $u$ we deduce 
$$\mathcal J(t,u)=\frac{\sinh^{2n-1}(kt)\cosh(kt)}{(kt)^{2n-1}}.$$
Thus, the volume of a compact convex domain $\Omega$ is given by
\begin{eqnarray}\label{volum}
\vol(\Omega)=\frac{1}{2nk^{2n}}\int_{S^{2n-1}}\sinh^{2n}(kl(u))dS_{2n-1},\end{eqnarray}
and the intrinsic volume of $\partial\Omega$ is given by
\begin{eqnarray}\label{area}
\vol(\partial\Omega)=\int_{S^{2n-1}}\frac{\sinh^{2n-1}(kl(u))\cosh(kl(u))}{k^{2n-1}\langle\partial_t,N\rangle}dS_{2n-1}\end{eqnarray}
where $N$ is the outward unit normal vector of $\partial\Omega$ and
$\partial_t$ the radial field from $p$. 

\medskip
In the proof of Theorem \ref{theorem2} we also use the following proposition about $\lambda$-convexity proved in \cite{borisenko.gallego.reventos}. The notion of $\lambda$-convexity gives some relations on how the boundary bends. Indeed, we have

\begin{proposition}[\cite{borisenko.gallego.reventos}]\label{teoremaBGR}
Let $M$ be a $(n+1)$-dimensional Hadamard manifold with sectional
curvature $K$ such that $-k_2^2\leq K\leq -k_1^2$. Let $\Omega$ be
a $\lambda$-convex domain with $\mathcal{C}^2$ boundary, $\lambda<
k_2$ and $O$ an interior point of $\Omega$. If $\varphi$ denotes
the angle of the normal to $\partial\Omega$ and the exterior
radial direction, when
$d(O,\partial\Omega)\leq(1/k_2)\arctanh(\lambda/k_2)$ we have
$$\cos\varphi\geq\frac{1}{k_2}\sqrt{\lambda^2\cosh^2k_2s-k_2^2\sinh^2k_2s}.$$
If $d(O,\partial\Omega)\geq(1/k_2)\arctanh(\lambda/k_2)$, then
$$\cos\varphi\geq\frac{\lambda}{k_{2}}.$$
\end{proposition}

\begin{proof}[Proof of Theorem 2]
Let $\varphi$ be the angle between the outward unit normal vector to $\partial\Omega(t)$ and the radial direction from a fixed point $O$ inside the convex set, then $0<\cos\varphi\leq 1$. Using (\ref{volum}), (\ref{area}) and the inequality $\sinh(x)\leq\cosh(x)$ for all $x\in\R$ we find
\begin{eqnarray}
&\frac{\displaystyle\vol(\Omega_t)}{\displaystyle\vol(\partial\Omega_t)}&=
\frac{\displaystyle\frac{1}{2nk^{2n}}\int_{S^{2n-1}}\sinh^{2n}(kl(u))dS_{2n-1}}
{\displaystyle\int_{S^{2n-1}}\frac{\sinh^{2n-1}(kl(u))\cosh(kl(u))}{k^{2n-1}\cos\varphi}dS_{2n-1}}\nonumber\\&&\leq
\frac{\displaystyle\int_{S^{2n-1}}\sinh^{2n}(kl(u))dS_{2n-1}}
{\displaystyle2nk\int_{S^{2n-1}}\sinh^{2n-1}(kl(u))\cosh(kl(u))dS_{2n-1}}\leq
\frac{1}{2nk}.\nonumber\end{eqnarray}

In order to obtain the lower bound we use that $\cos\varphi\geq{\lambda/2k}$ for a sufficiently large convex domain (see Proposition \ref{teoremaBGR}). Then,
\begin{eqnarray}
&\frac{\displaystyle\textrm{vol}(\Omega(t))}{\displaystyle\textrm{vol}(\partial\Omega(t))}&=
\frac{\displaystyle\frac{1}{2nk^{2n}}\int_{S^{2n-1}}\sinh^{2n}(kl(u))dS_{2n-1}}
{\displaystyle\int_{S^{2n-1}}\frac{\sinh^{2n-1}(kl(u))\cosh(kl(u))}{k^{2n-1}\cos\varphi}dS_{2n-1}}
\nonumber\\&&=
\frac{\displaystyle\int_{S^{2n-1}}\sinh^{2n-1}(kl(u))\cosh(kl(u))\tanh(kl(u))dS_{2n-1}}
{\displaystyle2nk\int_{S^{2n-1}}\frac{\sinh^{2n-1}(kl(u))\cosh(kl(u))}{\cos\varphi}dS_{2n-1}}
\nonumber\\&&\geq
\displaystyle\tanh(kr)\frac{\displaystyle\lambda\int_{S^{2n-1}}\sinh^{2n-1}(l(u))\cosh(l(u))dS_{2n-1}}
{\displaystyle4nk^2\int_{S^{2n-1}}\sinh^{2n-1}(l(u))\cosh(l(u))dS_{2n-1}}\nonumber\\&&=
\frac{\lambda}{4nk^2}\displaystyle\tanh(kr)\nonumber
\end{eqnarray} where $r$ is the distance between $O$
and the boundary of the convex domain. As $\tanh(kr)$
tends to 1 when $\Omega(t)$ expands to the whole space
we obtain 
$$\liminf_{t\rightarrow\infty}{\vol(\Omega(t))\over \vol(\partial\Omega(t))}\geq {1\over 4nk^2}.$$

\medskip
It remains to show the sharpness of the upper bound. Balls are $\lambda$-convex domains for every
$\lambda\in[0,k\coth(kr)]$ because normal curvatures are greater than or equal to $k\coth(kr)$ and families of balls are $\lambda$-convex for every $\lambda\in[0,k]$. As for a ball and for its boundary we have 
\begin{align*}\vol(B_{r})&=\frac{\sinh^{2n}(kr)\pi^n}{2nk^{2n}n!},\\\vol(\partial B_{r})&=\frac{\sinh^{2n-1}(kr)\cosh(kr)\pi^n}{k^{2n-1}n!},
\end{align*}
for a family of balls with the same
center we obtain
\begin{eqnarray*}
\limsup_{r\rightarrow\infty}\frac{\vol(B_r)}{\vol(\partial B_r)}&=&
\limsup_{r\rightarrow\infty}\frac{\sinh^{2n}(kr)\pi^n
n!}{2nk\sinh^{2n-1}(kr)\cosh(kr)\pi^n
n!}\\ &=&\limsup_{r\rightarrow\infty}\frac{\tanh(kr)}{2nk}=\frac{1}{2nk},
\end{eqnarray*}
which gives the upper bound on (\ref{fitesTeorema}).
\end{proof}

\begin{remark}The upper bound obtained here can be also deduced from \cite{thies}. Thies is interested in studying isoperimetric inequalities. He obtains an upper bound for the volume of a convex domain in a rank one symmetric space in terms of an integral over the boundary of the convex domain. The main tools he uses are the Riccati and Codazzi equations. Here we give a direct proof using an expression in polar coordinates for the volume of a convex domain and for the volume of its boundary. This expression is also used to obtain the lower bound. 
\end{remark}

\section{Some final remarks}\label{seccio5}
\emph{Generalization to non-compact rank one symmetric spaces.} An analogous result for Theorem \ref{theorem1} and Theorem \ref{theorem2} can be obtained in a non-compact rank one symmetric space using similar arguments to the ones in this paper (see \cite{berndt.quaternionic} and \cite{gray.vanhecke.corba} for an expression of the volume element in these spaces and the study of principal directions of a sphere). Non-compact rank one symmetric spaces are: real hyperbolic space, $\RH$, complex hyperbolic space, $\CH$, quaternionic hyperbolic space, $\QH$, and Cayley hyperbolic plane, $\OH$. They are Hadamard manifolds of real dimension $n$, $2n$, $4n$ and 16, respectively. Sectional curvature $K$ in $\CH$, $\QH$ and $\OH$ satisfies $-4k^2\leq K\leq -k^2$ and $\RH$ has constant sectional curvature equal to $-k^2$. We denote non-compact rank one symmetric spaces different from $\RH$ by $\KH$ with $\K=\{\C,\Q,\oc\}$ the associated real division algebra with real dimension $d$. We assume that $n=2$ if $\K=\oc$. 

Then Theorem \ref{theorem1} and Theorem \ref{theorem2} can be stated as Theorem A and Theorem B as follows. 

\begin{theoremA}In $\KH$ it can only exist families of compact $\lambda$-convex domains expanding over the whole $\KH$ if $\lambda\leq k$.
\end{theoremA}

\begin{theoremB}Let $\{\Omega(t)\}_{t\in\R^+}$ be a piecewise $\mathcal{C}^2$ family of
$\lambda$-convex compact domains, $0\leq \lambda\leq k$, expanding over the
whole space $\KH$, $n\geq 2$. Then
$$\frac{\lambda}{2dnk^2}\leq
\liminf_{t\rightarrow\infty}\frac{\emph{vol}(\Omega(t))}{\emph{vol}(\partial\Omega(t))}\leq
\limsup_{t\rightarrow\infty}\frac{\emph{vol}(\Omega(t))}{\emph{vol}(\partial\Omega(t))}\leq
\frac{1}{(dn+d-2)k}.$$ Moreover, the upper bound is sharp.
\end{theoremB}

\emph{Quotient for the modified tubes.} Gray in \cite{gray.vanhecke.corba} obtains the following formula for the volume of a tube about a curve $\sigma$ of length $L$ in complex hyperbolic space
$$\vol(\tau_L(\sigma,r))=\frac{L \vol(S^{2n-2})}{(4k^2)^{n-1}}\int_0^r\sinh^{2n-2}(ks)\left(1+\frac{2n}{2n-1}\sinh^2(ks)\right)ds.$$
If we fix a geodesic line $\gamma\subset\CH$ and we consider the family $\{\tau_{L}(\gamma,r)\}_{\{r,L\}}$ of modified tubes we obtain a family of compact convex domains expanding over the whole $\CH$. To determine the value of \begin{align*}&\limsup_{r\rightarrow\infty, L\rightarrow\infty}\frac{\textrm{vol}(\tau_L(\gamma,r))}{\textrm{vol}(\partial \tau_L(\gamma,r))}=\limsup_{r\rightarrow\infty, L\rightarrow\infty}\frac{\textrm{vol}(\tau(\gamma,r))+\vol(B_{r})}{\textrm{vol}(\partial \tau(\gamma,r))+\vol(\partial B_{r})}\\&=\limsup_{r\rightarrow\infty, L\rightarrow\infty}\frac{\textrm{vol}(\tau(\gamma,r))}{\textrm{vol}(\partial \tau(\gamma,r))+\vol(\partial B_{r})}+\limsup_{r\rightarrow\infty, L\rightarrow\infty}\frac{\vol(B_{r})}{\textrm{vol}(\partial \tau(\gamma,r))+\vol(\partial B_{r})},\end{align*} we calculate the next four limits. In the first one we use L'H\^{o}pital's rule
\begin{align*}\limsup_{r\rightarrow\infty,
L\rightarrow\infty}\frac{\vol(\partial\tau(\gamma,r))}{\vol(\tau(\gamma,r))}=2nk, \quad
&\limsup_{r\rightarrow\infty,L\rightarrow\infty}
\frac{\vol(\partial B_{r})}{\vol(\tau(\gamma,r))}=0,
\\\limsup_{r\rightarrow\infty,L\rightarrow\infty}
\frac{\vol(\partial\tau(\gamma,r))}{\vol(B_{r})}=\infty,\quad&
\limsup_{r\rightarrow\infty,L\rightarrow\infty}
\frac{\vol(\partial B_{r})}{\vol(B_{r})}=2nk.
\end{align*}
So that, 
\begin{align*}\limsup_{r\rightarrow\infty, L\rightarrow\infty}\frac{\textrm{vol}(\tau_L(\gamma,r))}{\textrm{vol}(\partial \tau_L(\gamma,r))}
=\frac{1}{2nk}\end{align*}
which is also equal to the value of the upper bound in (\ref{fitesTeorema}). 
The value of this limit does not depend on the relation between the growth of the length of the segment and the radius of the tube. 

By an analogous but longer calculation it can be also proved that the limit of the quotient volume/area for the family of compact convex domains constructed as in Remark \ref{altresTubes} is also $1/2nk$. Thus, it does not depend on the relation between the radius of the tube and the radius of the sphere. 

\emph{Lower bound.} To complete the study of this asymptotic behaviour in complex hyperbolic space it remains to decide if the lower bound in Theorem \ref{theorem2} is sharp. In the study of the analogous problem in real hyperbolic space it is given another family of convex domains expanding over the whole space which tends to the lower bound (see \cite{solanes.tesi}). Let us describe this family briefly. Consider a ball of radius $r>0$ centered at a point $p\in\RH$. Let $p_{1}$ and $p_{2}$ be the endpoints of a segment of length $2R$ ($R>r$) with $p$ its midpoint. Take the convex hull $Q(r,R)$ of the ball and the points $p_{1}$ and $p_{2}$. The tube of radius $\epsilon$ about $Q(r,R)$ is a $\tanh(\epsilon)$-convex hypersurface. When $r$ and $R$ tend to infinity we obtain a family of compact  $\tanh(\epsilon)$-convex domains. The quotient between its volume and its area tends to the lower bound in $\RH$.

If we try to construct a family of domains in a similar way we get into trouble due to the different trigonometric formulas of the trigonometry in the complex hyperbolic space.

\end{document}